\documentclass[12pt]{article}
\usepackage{amsthm}
\usepackage{amsmath}
\usepackage{amsfonts}
\usepackage{graphicx}
\usepackage{latexsym}
\usepackage{amssymb}

\newcommand{\CMRD}{\C^{\textmd{MRD}}}

\newcommand{\deff}{\mbox{$\stackrel{\rm def}{=}$}}
\newcommand{\Span}[1]{{\left\langle {#1} \right\rangle}}

\newcommand{\field}[1]{\mathbb{#1}}
\newcommand{\B}{\field{B}}
\newcommand{\C}{\field{C}}

\newcommand{\F}{\field{F}}
\newcommand{\V}{\field{V}}

\newcommand{\dS}{\field{S}}
\newcommand{\dT}{\field{T}}

\newcommand{\cC}{{\cal C}}

\newcommand{\cF}{{\cal F}}

\newcommand{\cG}{{\cal G}}

\newtheorem{theorem}{Theorem}
\newtheorem{lemma}{Lemma}
\newtheorem{remark}{Remark}
\newtheorem{cor}{Corollary}
\newtheorem{example}{Example}

\begin{document}

\bibliographystyle{plain}

\title{
\begin{center}
Partial $k$-Parallelisms in Finite Projective Spaces
\end{center}
}
\author{
{\sc Tuvi Etzion}\thanks{Department of Computer Science, Technion,
Haifa 32000, Israel, e-mail: {\tt etzion@cs.technion.ac.il}.} }

\maketitle

\begin{abstract}
In this paper we consider the following question. What is the maximum number of
pairwise disjoint $k$-spreads which exist in PG($n,q$)?
We prove that if $k+1$ divides $n+1$ and $n>k$ then
there exist at least two disjoint $k$-spreads in PG($n,q$)
and there exist at least
$2^{k+1}-1$ pairwise disjoint $k$-spreads in PG($n,2$).
We also extend the known results on parallelism in a projective
geometry from which the points of a given subspace were removed.
\end{abstract}

\vspace{0.5cm}


\footnotetext[1] { This research was supported in part by the Israeli
Science Foundation (ISF), Jerusalem, Israel, under
Grant 10/12.}

%%%%%%%%%%%%%%%%%%%%%%%%%%%%%%%%%%%%%%%%%%%%%%%%%%%%%%%%%%%%%%%%%%%%%%
%%%%%%%%%%%%%%%%%%%%%%%%%%%%%%%%%%%%%%%%%%%%%%%%%%%%%%%%%%%%%%%%%%%%%%
%%%%%%%%%%%%%%%%%%%%%%%%%%%%%%%%%%%%%%%%%%%%%%%%%%%%%%%%%%%%%%%%%%%%%%
\newpage
\section{Introduction}
\label{sec:introduction}

A $k$-spread in the $n$-dimensional projective space of finite
order $q$, namely PG($n,q$) is a set $\dS$ of $k$-dimensional subspaces
(henceforth called $k$-subspaces) in which
each point of PG($n,q$) is contained in exactly one element of $\dS$.
A necessary and sufficient condition that
a $k$-spread exists in PG($n,q$) is that
$k+1$ divides $n+1$. The size of a $k$-spread
in PG($n,q$) is $\frac{q^{n+1}-1}{q^{k+1}-1}$.
$k$-spreads were extensively studied since they
have many applications in projective geometry, e.g.~\cite{BeUe91,Lun99}.

Parallelism is a well known concept in combinatorial designs.
A \emph{parallel class} in a block design, is a set of blocks which
partition the set of points of the design.
Spreads are also a type of combinatorial design on which a parallelism
can be defined~\cite{Joh10}.
A $k$-spread is called a parallel class as it partitions
the set of all the points of PG($n,q$). A $k$-parallelism in
PG($n,q$) is a partition of the $k$-subspaces
of PG($n,q$) into pairwise disjoint $k$-spreads. Some 1-parallelisms
of PG($n,q$) are known for many years. For $q=2$ and
odd $n$ there is a 1-parallelism in PG($n,2$). Such a parallelism
was found in the context of Preparata codes and it is known that
many such parallelisms exist~\cite{Bak76,BLW}. For any other power
of a prime $q$, if $n=2^i-1$, $i \geq 2$, then
a 1-parallelism was shown in~\cite{Beu74}.
In the last forty years no new parameters for
1-parallelisms were shown until recently, when
a 1-parallelism in PG(5,3) was proved to exist in~\cite{EtVa12}.
A $k$-parallelism for $k > 1$ was not known until a 2-parallelism
in PG($5,2$) was found by~\cite{Sar02}.

The difficulty to find new parameters for 1-parallelisms and $k$-parallelisms
motivates the following question. What is the maximum number
of pairwise disjoint $k$-spreads that exist in PG($n,q$)?
Beutelspacher~\cite{Beu90} has proved that if $n$ is odd then there exist
$q^{2 \lfloor \log n \rfloor} + \cdots +q+1$ pairwise disjoint 1-spreads in PG($n,q$).
In general we don't have a proof for the following most simple question.
Given $q$, $n$, and $k$, such that $k+1$ divides $n+1$
and $n>k$, do there exist two disjoint
$k$-spreads in PG($n,q$)? In this paper we will give a positive
answer for this question. Moreover, we will prove that there exist
at least $2^{k+1}-1$ pairwise disjoint $k$-spreads in PG($n,2$)
if $k+1$ divides $n+1$ and $n>k$.

One of the main tools for our constructions will come from coding theory.
It will based on error-correcting codes in the Grassmannian space
which are constructed by lifting matrices of error-correcting codes
in the rank-metric. This method is well documented, e.g.~\cite{EtSi09,EtSi13,SKK08}.
The interest in such construction came as result of
a new application of such codes
in random network coding~\cite{KK}.

The rest of this paper is organized as follows.
In Section~\ref{sec:conf} we will present the two
equivalent ways to handle subspaces, in projective geometry
and in the Grassmannian. We will explain the method which transfers
matrices into subspaces and rank-metric codes into Grassmannian codes.
We will present some basic results and connect them into the
theory of projective geometry in general and the theory of spreads
in particular.
In Section~\ref{sec:n=2k} we present a construction which produces
$2^{k+1}-1$ disjoint $k$-spreads in PG($2k+1,2$). In Section~\ref{sec:any_q}
we prove that for a general $q$ there exist at least two disjoint
$k$-spreads in PG($2k+1,q$). In Section~\ref{sec:recursive}
we present a recursive construction to obtain two disjoint $k$-spreads
in PG($n,q$) if $k+1$ divides $n+1$ and $n>k$; and $2^{k+1}-1$ pairwise disjoint
$k$-spreads in PG($n,2$) if $k+1$ divides $n+1$ and $n>k$.
The construction will be based on a design called subspace
transversal design which will be defined.
It will lead to parallelisms in partial sets of PG($n,q$)
which will be used in the recursive construction.
Conclusions and problems for future research are given
in Section~\ref{sec:conclusion}.

\section{Representation of subspaces, codes, and spreads}
\label{sec:conf}

The projective geometry PG($n,q$) consists of $\frac{q^{n+1}-1}{q-1}$
points and $\frac{(q^{n+1}-1)(q^n-1)}{(q^2-1)(q-1)}$ lines.
The points are represented by a set of
nonzero elements from $\F_q^{n+1}$, of maximum size,
in which each two elements are linearly independent.
Each element $x$ of these $\frac{q^{n+1}-1}{q-1}$
elements represents $q-1$
elements of $\F_q^{n+1}$ which are the multiples of $x$ by the nonzero
elements of $\F_q$. A line in PG($n,q$) consists of $q+1$ points.
Given two distinct points $x$ and $y$, there
is exactly one line which contains these two points.
This line contains $x$ and $y$ and the $q-1$ points
of the form $\gamma x +y$, where $\gamma \in \F_q \setminus \{ 0 \}$.
A point is a 0-subspace in PG($n,q$), a line is a 1-subspace in PG($n,q$),
and a $k$-subspace is constructed by taking a
$(k-1)$-space $Y$ and a point $x$ not on $Y$ and all points that
are constructed by a linear combination of $x$ with any set of points
from $Y$.

The \emph{Grassmannian} $\cG_q(n,k)$ consists of all the $k$-dimensional
subspaces of $\F_q^n$. Clearly, a $k$-dimensional subspace from
$\cG_q(n,k)$ is a $(k-1)$-subspace of PG($n-1,q$). Extensive research
has been done on the Grassmannian in the past few years. The motivation
for this research is the application of codes in the Grassmannian for
error-correction in random network coding found recently
by Koetter and Kschischang~\cite{KK}.

A subset $\C$ of $\cG_q(n,k)$ is called an  $(n,M,d,k)_q$ \emph{constant
dimension code} if it has size $M$ and minimum subspace distance $d$,
where the distance function in $\cG_q(n,k)$ is defined by
$$
d_S (X,\!Y) \deff 2k
-2 \dim\bigl( X\, {\cap}Y\bigr),
$$
for any two subspaces $X$ and $Y$ in $\cG_q(n,k)$.

Two $k$-dimensional subspaces in $\cG_q(n,k)$ are called
disjoint if their intersection is the null space.
A \emph{spread} in $\cG_q(n,k)$ is a set $\dS$ of pairwise
disjoint $k$-dimensional subspaces, such that each
nonzero element of $\F_q^n$ is contained in exactly one
element of $\dS$. Clearly, such a spread is a $(k-1)$-spread
in PG($n-1,q$). Hence, a spread in $\cG_q(n,k)$ exists
if and only if $k$ divides $n$. A~set of $M$ pairwise
disjoint spreads in $\cG_q(n,k)$ is a set of $M$ pairwise
disjoint $(k-1)$-spreads in PG($n-1,q$). Henceforth, our discussion
will be in terms of $k$-dimensional subspaces of $\cG_q(n,k)$
and will be translated into related results
in terms of subspaces in projective geometry. The reason is that
some of the new developed theory for constant dimension codes
will serve as the building blocks for our constructions and results.

One of the main constructions for constant dimension codes
is based on rank-metric codes.
For two $k \times \ell$ matrices $A$ and $B$ over $\F_q$ the {\it
rank distance} is defined by
$$
d_R (A,B) \deff \text{rank}(A-B)~.
$$
A  $[k \times \ell,\varrho,\delta]_q$ {\it rank-metric code} $\cC$
is a linear code, whose codewords are $k \times \ell$ matrices
over~$\F_q$; they form a linear subspace with dimension $\varrho$
of $\F_q^{k \times \ell}$, and for each two distinct codewords
$A$~and~$B$ we have that $d_R (A,B) \geq \delta$ (clearly,
$\delta \leq \min \{k , \ell \}$). For a $[k \times
\ell,\varrho,\delta]_q$ rank-metric code~$\cC$ it was proved
in~\cite{Del78,Gab85,Rot91} that
\begin{equation}
\label{eq:MRD} \varrho \leq \text{min}
\left\{k(\ell-\delta+1),\ell(k-\delta+1) \right\}~.
\end{equation}
This bound is attained for all possible parameters and the codes
which attain it are called {\it maximum rank distance} codes (or
MRD codes in short).

There is a close connection between constant dimension codes and
rank-metric codes~\cite{EtSi09,SKK08}.
Let $A$ be a $k \times \ell$ matrix over~$\F_q$ and let $I_k$ be the
$k \times k$ identity matrix. The matrix $[ I_k ~ A ]$ can be
viewed as a generator matrix of a $k$-dimensional subspace of
$\F_q^{k+\ell}$, and it is called the \emph{lifting} of
$A$~\cite{SKK08}.
\begin{example}
Let $A$ and $[I_3 ~A]$ be the following matrices over $\F_2$
$$A=\left( \begin{array}{ccc}
1& 1 & 0\\
0& 1 & 1\\
0& 0 & 1
\end{array}
\right) ~, ~~ [I_3 ~ A] =\left( \begin{array}{cccccc}
1&0&0&1& 1 & 0\\
0&1&0&0& 1 & 1\\
0&0&1&0& 0 & 1
\end{array}
\right),
$$
then the subspace obtained by the lifting of $A$ is given by the
following $8$ vectors:
$$(1 0 0 1 1 0),
(0 1 0 0 1 1), (0 0 1 0 0 1),(1 1 0 1 0 1),$$
$$(1 0 1 1 1 1),(0 1 1 0 1 0),
(1 1 1 1 0 0),(0 0 0 0 0 0).
$$
\end{example}

A constant dimension code $\mathbb{C}$ such that all its codewords
are lifted codewords of an MRD code is called a \emph{lifted MRD
code}~\cite{SKK08}. This code will be denoted by $\CMRD$.
A lifted MRD code constructed from $[k \times (n-k),
(n-k)(k-\delta +1),\delta ]_q$ MRD code will be called
an $(n,k,\delta)_q$ $\CMRD$.

\begin{theorem}\cite{SKK08}
\label{trm:param lifted MRD} If $\cC$ is a $[k \times (n-k),
(n-k)(k-\delta +1),\delta ]_q$ MRD code then $(n,k,\delta)_q$ $\C^{\textmd{MRD}}$ is
an $(n,q^{(n-k)(k-\delta+1)}, 2\delta, k)_{q}$ code.
\end{theorem}

\begin{remark}
The parameters of the $[k \times (n-k), (n-k)(k-\delta +1),\delta ]_q$
MRD code $\cC$ given in Theorem~\ref{trm:param lifted MRD} imply
that $k \leq n-k$, by (\ref{eq:MRD}).
\end{remark}

Let $\V^{(n,k)}$ be the set of nonzero vectors of $\F_q^n$ whose
first $k$ entries form a nonzero vector.
The following results were proved in~\cite{EtSi13}.

\begin{lemma}
\label{lem:resolv}
The codewords of an $(n,k,\delta)_q$ $\CMRD$
can be partitioned into $q^{(n-k)(k-\delta)}$
sets, called parallel classes, each one of size $q^{n-k}$, such
that in each parallel class each element of $\V^{(n,k)}$ is
contained in exactly one codeword.
\end{lemma}
\begin{cor}
\label{cor:MRD_parts}
The codewords of an $(n,k,\delta)_q$ $\CMRD$
can be partitioned into $q^{(n-k)(k-\delta)}$
codes, each one is an $(n,q^{n-k},2k,k)_q$ code.
\end{cor}

%Our next two definitions are given in this section only for $q=2$.
%They will be modified and generalized in Section~\ref{sec:any_q}.
%But, for better understanding and since the next section consider
%only the case where $q=2$ we restrict our discussion at this point for
%this case.

For a given $x \in \F_2^k$, let $\V_x^{(n,k)}$ denote the set nonzero
vectors in $\F_2^n$ whose first $k$ entries form the vector $x$.
In the sequel, let ${\bf 0}$ denote the all-zero vector.

In the sequel we will represent nonzero elements of the finite field
$\F_{2^m}$ in two different ways. The first one is by $m$-tuples over $\F_2$
(in other words, $\F_{2^m}$ is represented by $\F_2^m$)
and the second one is by powers of a primitive element $\alpha$ in $\F_{2^m}$.
We will not distinguish between these two isomorphic representations. When an $m$-tuple $z$
over $\F_2$ will be multiplied by an element $\beta \in \F_{2^m}$ we
will view $z$ as an element in $\F_{2^m}$ and the result will be
an element in $\F_{2^m}$ which is also represented by an $m$-tuple
over $\F_2$ (an element in $\F_2^m$). Also, when we write
$\V_\gamma^{(n,k)}$, where $\gamma \in \F_{2^k}$, it is
the same as writing $\V_x^{(n,k)}$, $x \in \F_2^k$, where $x$
is the binary $k$-tuple which represents~$\gamma$.
Therefore, vectors can be represented by powers of primitive
elements in the related finite field. We will use this notation
in some cases.

For a set $S \subseteq \F_2^m$ and a nonzero element $\beta\in \F_{2^m}$,
we define $\beta S \deff \{ \beta x ~:~ x \in S \}$. We
note that we can take the set $S$ to be a subspace. By using the
Singer cycle subgroup~\cite{Kan80} it is observed that
if $X$ is a $k$-dimensional subspace of $\F_2^m$ and $\beta$ is a
nonzero element of $\F_{2^m}$ then~$\beta X$ is also a
$k$-dimensional subspace of $\F_2^m$. This property will be used throughout
the paper.

\section{A construction for $q=2$ and $n=2k$}
\label{sec:n=2k}

Recall that the vectors of $\F_2^{2k} \setminus \{ {\bf 0} \}$ are partitioned into $2^k$ parts,
$\V_x^{(2k,k)}$, $x \in \F_2^k$.
Let $\V_{\bf 0}$ denote the $k$-dimensional subspace spanned by $\V_{\bf 0}^{(2k,k)}$.

Consider $k$-dimensional subspaces from $\cG_2(2k,k)$ of three types:
\begin{enumerate}
\item A $k$-dimensional subspace $Y \in \cG_2(2k,k)$ is of Type A if
for each $x \in \F_2^k \setminus \{ {\bf 0} \}$, $Y$ contains
exactly one vector from $\V_x^{(2k,k)}$, and $Y$ does not contain
any vector from $\V_{\bf 0}^{(2k,k)}$.

\item A $k$-dimensional subspace $Y \in \cG_2(2k,k)$ is of Type B if
$Y$ contains exactly one vector from $\V_{\bf 0}^{(2k,k)}$.

\item A $k$-dimensional subspace $Y \in \cG_2(2k,k)$ is of Type C if all the nonzero vectors
of $Y$ are contained in $\V_{\bf 0}^{(2k,k)}$, i.e. $Y = \V_{\bf 0}$.
\end{enumerate}

One can readily verify that
\begin{lemma}
\label{lem:TypeB_structure}
If $Z$ is a $k$-dimensional subspace of Type B then $Z$ has the structure
$$\{ ( {\bf 0},{\bf 0}),( {\bf 0},z),(x_0,y_0),(x_0,y_1),(x_1,y_2),(x_1,y_3),
\ldots,(x_{2^{k-1}-2},y_{2^k-3}),(x_{2^{k-1}-2},y_{2^k-2}) \},$$
where $\{ {\bf 0},x_0,x_1,\ldots,x_{2^{k-1}-2} \}$ is a $(k-1)$-dimensional
subspace of $\F_2^{2k}$ and for each $i$, $0 \leq i \leq 2^{k-1}-2$, we have
$z= y_{2i} + y_{2i+1}$.
\end{lemma}

For completeness, even so it is not necessary for our discussion,
we give the following lemma without a proof (this is left for the
interested reader).
\begin{lemma}
$~$
\begin{itemize}
\item There exist exactly $2^{k^2}$ distinct $k$-dimensional subspaces of Type A.

\item There exist exactly $(2^k-1)^2 2^{(k-1)^2}$ distinct $k$-dimensional subspaces of Type B.

\item There exists exactly one $k$-dimensional subspace of Type C.
\end{itemize}
\end{lemma}

Our construction which follows will yield
$2^k-1$ pairwise disjoint spreads in $\cG_2(2k,k)$. Each spread will consist
of exactly $2^k-1$ subspaces of Type B and exactly two subspaces of Type~A.
In the construction, a $k$-dimensional subspace $Z$ of $\F_2^{2k}$ will be represented as
$$Z=\{ ( {\bf 0},{\bf 0}),(x_0,y_0),(x_1,y_1),\ldots,(x_{2^k-2},y_{2^k-2}) \}~,$$
where $x_i,y_i \in \F_2^k$ and
$y_i \neq {\bf 0}$ if $x_i = {\bf 0}$, $0 \leq i \leq 2^k-2$.

Let $\C_0$ be a $(2k,k,k-1)_2$ $\CMRD$, i.e. a $(2k,2^{2k},2(k-1),k)_2$ code.
By Corollary~\ref{cor:MRD_parts},
$\C_0$ can be partitioned into $2^k$ codes, each one is a
$(2k,2^k,2k,k)_2$ code.
Each one of these $2^k$ codes can be completed to a spread
if we add $\V_{\bf 0}$ to the code.
$\C_0$ is constructed from a linear rank-metric code~$\cC$ and therefore
one of its codewords is the $k$-dimensional subspace
$\{ ({\bf 0},{\bf 0}),(x_0,{\bf 0}),(x_1,{\bf 0}),\ldots,(x_{2^k-2},{\bf 0}) \}$.
Since the minimum subspace distance of~$\C_0$ is $2(k-1)$,
it follows that for each other codeword
$\{ ( {\bf 0},{\bf 0}),(x_0,y_0),(x_1,y_1),\ldots,(x_{2^k-2},y_{2^k-2}) \}$ of $\C_0$,
at most one of $y_i$'s is the all-zero vector. Therefore,

\begin{lemma}
\label{lem:Cstruct}
The code $\C_0$ can be partitioned into $2^k$ $(2k,2^k,2k,k)_2$ codes,
for which, each one which does not contain the codeword
$\{ ({\bf 0},{\bf 0}),(x_0,{\bf 0}),(x_1,{\bf 0}),\ldots,(x_{2^k-2},{\bf 0}) \}$,
contains exactly $2^k-1$ codewords of the form
$$\{ ( {\bf 0},{\bf 0}),(x_0,y_0),(x_1,y_1),\ldots,(x_{2^k-2},y_{2^k-2}) \}~,$$
in which exactly one of the $y_i$'s is the all-zero vector.
\end{lemma}

\begin{cor}
\label{cor:Cstruct}
There exists a $(2k,2^k+1,2k,k)_2$ code which contains $\V_{\bf 0}$
as a codeword and for each codeword
$\{ ( {\bf 0},{\bf 0}),(x_0,y_0),(x_1,y_1),\ldots,(x_{2^k-2},y_{2^k-2}) \}$
at most one of the $y_i$'s is the all-zero vector.
\end{cor}

Let $\C$ be a $(2k,2^k+1,2k,k)_2$ code as described in
Corollary~\ref{cor:Cstruct}, i.e. it contains $\V_{\bf 0}$
as a codeword and for each codeword
$\{ ( {\bf 0},{\bf 0}),(x_0,y_0),(x_1,y_1),\ldots,(x_{2^k-2},y_{2^k-2}) \}$
at most one of the~$y_i$'s is the all-zero vector.
Let $\overleftrightarrow{\C}$ be the $(2k,2^k+1,2k,k)_2$ code
obtained from $\C$ as follows
$$
\overleftrightarrow{\C} \deff \Bigl\{ \{ ( {\bf 0},{\bf 0}),(x_0,y_0),(x_1,y_1),\ldots,(x_{2^k-2},y_{2^k-2}) \} ~:~
\{ ( {\bf 0},{\bf 0}),(y_0,x_0),(y_1,x_1),\ldots,(y_{2^k-2},x_{2^k-2}) \} \in \C  \Bigr\} .
$$

Henceforth, let $\alpha$ be a primitive element in the field $\F_{2^k}$.
As a consequence of Lemma~\ref{lem:Cstruct} and
Corollary~\ref{cor:Cstruct} we have
\begin{lemma}
\label{lem:C_two_types}
The code $\overleftrightarrow{\C}$ is a spread, in $\cG_2(2k,k)$, which consists
of exactly $2^k-1$ subspaces of Type B and exactly two subspaces of Type A.
One of the two subspaces of Type A has the form
$\{ ( {\bf 0},{\bf 0}),(x_0,{\bf 0}),(x_1,{\bf 0}),\ldots,(x_{2^k-2},{\bf 0}) \}$.
\end{lemma}

Now, we are in a position to define $2^k-1$ pairwise disjoint spreads in
$\cG_2(2k,k)$. For our first spread $\dS_0$ defined as follows,
we distinguish between two cases:

\noindent
{\bf Case 1:} If there is no subspace in $\overleftrightarrow{\C}$
of the form
$\{ ( {\bf 0},{\bf 0}),(\alpha^0, \alpha^j ),(\alpha^1, \alpha^{j+1} ),\ldots,(\alpha^{2^k-2}, \alpha^{j+2^k-2}) \}$,
for any $j$, $0 \leq j \leq 2^k-2$, then
$$
\dS_0 \deff \Bigl\{ \{ ( {\bf 0},{\bf 0}),(\alpha^0,y_0 +\alpha^0),(\alpha^1,y_1 + \alpha^1),
(\alpha^2 , y_2 + \alpha^2 ),\ldots,(\alpha^{2^k-2},y_{2^k-2} +\alpha^{2^k-2}) \} ~:~
$$
$$
\{ ( {\bf 0},{\bf 0}),(\alpha^0,y_0),(\alpha^1,y_1),(\alpha^2,y_2),\ldots,(\alpha^{2^k-2},y_{2^k-2}) \} \in \overleftrightarrow{\C} \Bigr\}
$$
$$
\cup ~ \Bigl\{ \{ ( {\bf 0},{\bf 0}),( {\bf 0},z),(\alpha^{i_0},y_0+\alpha^{i_0}),(\alpha^{i_0},y_1+\alpha^{i_0}),
\ldots,(\alpha^{i_{2^{k-1}-2}},y_{2^k-4}+\alpha^{i_{2^{k-1}-2}}),(\alpha^{i_{2^{k-1}-2}},y_{2^k-3}+\alpha^{i_{2^{k-1}-2}}) \} ~:~
$$
$$
\{ ( {\bf 0},{\bf 0}),( {\bf 0},z),(\alpha^{i_0},y_0),(\alpha^{i_0},y_1),(\alpha^{i_1},y_2),(\alpha^{i_1},y_3),
\ldots,(\alpha^{i_{2^{k-1}-2}},y_{2^k-4}),(\alpha^{i_{2^{k-1}-2}},y_{2^k-3}) \} \in \overleftrightarrow{\C} \Bigr\} ~.
$$

\noindent
{\bf Case 2:} If there exists a subspace in $\overleftrightarrow{\C}$
of the form
$\{ ( {\bf 0},{\bf 0}),(\alpha^0, \alpha^j ),(\alpha^1, \alpha^{j+1} ),\ldots,(\alpha^{2^k-2}, \alpha^{j+2^k-2}) \}$
for some $j$, $0 \leq j \leq 2^k-2$, then
$$
\dS_0 \deff \Bigl\{ \{ ( {\bf 0},{\bf 0}),(\alpha^0,y_0 +\alpha^0),(\alpha^1,y_1 + \alpha^2),
(\alpha^2 ,y_2 +\alpha^4 ),\ldots,(\alpha^{2^k-2},y_{2^k-2} +\alpha^{2^k-3}) \} ~:~
$$
$$
\{ ( {\bf 0},{\bf 0}),(\alpha^0,y_0),(\alpha^1,y_1),(\alpha^2,y_2),\ldots,(\alpha^{2^k-2},y_{2^k-2}) \} \in \overleftrightarrow{\C} \Bigr\}
$$
$$
\cup ~\Bigl\{ \{ ( {\bf 0},{\bf 0}),( {\bf 0},z),(\alpha^{i_0},y_0+\alpha^{2 \cdot i_0}),(\alpha^{i_0},y_1+\alpha^{2 \cdot i_0}),
\ldots,(\alpha^{i_{2^{k-1}-2}},y_{2^k-4}+\alpha^{2 \cdot i_{2^{k-1}-2}}),(\alpha^{i_{2^{k-1}-2}},y_{2^k-3}+\alpha^{2 \cdot i_{2^{k-1}-2}}) \} ~:~
$$
$$
\{ ( {\bf 0},{\bf 0}),( {\bf 0},z),(\alpha^{i_0},y_0),(\alpha^{i_0},y_1),
\ldots,(\alpha^{i_{2^{k-1}-2}},y_{2^k-4}),(\alpha^{i_{2^{k-1}-2}},y_{2^k-3}) \} \in \overleftrightarrow{\C} \Bigr\}~.
$$

The following two lemmas can be easily verified.

\begin{lemma}
\label{lem:add2_1}
If $\{ ({\bf 0},{\bf 0}),( {\bf 0},z), (\alpha^{i_0}, y_0), (\alpha^{i_0},y_1) , \ldots ,
(\alpha^{i_{2^{\ell-1}-2}},y_{2^\ell-4}),(\alpha^{i_{2^{\ell-1}-2}},y_{2^\ell-3})   \}$ is
an $\ell$-dimensional subspace and $\{ ({\bf 0},{\bf 0}), (\alpha^{i_0}, v_0), \ldots ,
(\alpha^{i_{2^{\ell-1}-2}},v_{2^{\ell-1}-2})  \}$ is an $(\ell-1)$-dimensional
subspace then
$$\{ ({\bf 0},{\bf 0}),( {\bf 0},z), (\alpha^{i_0}, y_0+v_0), (\alpha^{i_0},y_1+v_0) , \ldots ,
(\alpha^{i_{2^{\ell-1}-2}},y_{2^\ell-4}+v_{2^{\ell-1}-2}),(\alpha^{i_{2^{\ell-1}-2}},y_{2^\ell-3}+v_{2^{\ell-1}-2})   \}$$
is an $\ell$-dimensional subspace.
\end{lemma}

\begin{lemma}
\label{lem:add2_2}
If $\{ ({\bf 0},{\bf 0}), (\alpha^{i_0}, y_0), (\alpha^{i_1},y_1) , \ldots ,
(\alpha^{i_{2^\ell-2}},y_{2^\ell-2})  \}$ and
$\{ ({\bf 0},{\bf 0}), (\alpha^{i_0}, v_0), (\alpha^{i_1},v_1) , \ldots ,
(\alpha^{i_{2^\ell-2}},v_{2^\ell-2})  \}$ are
two distinct $\ell$-dimensional subspaces then
$$\{ ({\bf 0},{\bf 0}), (\alpha^{i_0}, y_0+v_0), (\alpha^{i_1},y_1+v_1) , \ldots ,
(\alpha^{i_{2^\ell-2}},y_{2^\ell-2}+v_{2^\ell-2})  \}$$
is an $\ell$-dimensional subspace.
\end{lemma}

\begin{lemma}
$\dS_0$ is a spread.
\end{lemma}
\begin{proof}
By Lemma~\ref{lem:C_two_types}, $\overleftrightarrow{\C}$ is a spread.
By Lemmas~\ref{lem:add2_1} and~\ref{lem:add2_2}, the elements defined in
$\dS_0$ are $k$-dimensional subspaces. It is easy to verify by the definition
of $\dS_0$ that if $X$ and $Y$ are two disjoint $k$-dimensional subspaces
of $\overleftrightarrow{\C}$ then their related $k$-dimensional subspaces
$X'$ and $Y'$, respectively (constructed from $X$ and $Y$, respectively)
in $\dS_0$ are also disjoint. Therefore, $\dS_0$ is a spread.
\end{proof}

By Lemma~\ref{lem:C_two_types} and by the definition of $\dS_0$ we have that
\begin{lemma}
\label{lem:S0_two_types}
The spread $\dS_0$ consists
of exactly $2^k-1$ subspaces of Type B and exactly two subspaces of Type A.
\end{lemma}

By Lemma~\ref{lem:C_two_types} and by the definition of $\dS_0$ we also have that
\begin{lemma}
\label{lem:S0_typeA}
No subspace in $\dS_0$ has the form
$\{ ( {\bf 0},{\bf 0}),(x_0,{\bf 0}),(x_1,{\bf 0}),\ldots,(x_{2^k-2},{\bf 0}) \}$.
At most one of the subspaces of Type A in $\dS_0$ has the form
$\{ ( {\bf 0},{\bf 0}),(\alpha^0, \alpha^j ),(\alpha^1, \alpha^{j+1} ),\ldots,(\alpha^{2^k-2}, \alpha^{j+2^k-2}) \}$,
for some $j$, $0 \leq j \leq 2^k-2$.
\end{lemma}

\begin{lemma}
Let $X_1,X_2,\ldots , X_{2^\ell-1}$ be $2^\ell-1$ $(\ell-1)$-dimensional
subspaces of $F_2^\ell$. If each nonzero element of $\F_2^\ell$ is contained
in exactly $2^{\ell-1}-1$ subspaces of these $2^\ell-1$ subspaces then
$X_1,X_2,\ldots , X_{2^\ell-1}$ are distinct $(\ell-1)$-dimensional subspaces.
\end{lemma}
\begin{proof}
First note the $\dim (X_i \cap X_j) = \ell-2$ for $1 \leq i < j \leq 2^\ell-1$.
Hence, for any given $r$, $1 \leq r \leq 2^\ell-1$,
\begin{equation}
\label{eq:cont1}
\sum_{i=1}^{2^\ell-1} | X_i \cap X_r | = 2^\ell-1 + \lambda_r (2^{\ell-1}-1) + (2^\ell-1 -\lambda_r) (2^{\ell-2}-1)~,
\end{equation}
where $\lambda_r$ is the number of subspaces in $X_1,X_2,\ldots , X_{2^\ell-1}$ which equals $X_r$.
On the other hand, since each nonzero element of $X_r$ is contained
in exactly $2^{\ell-1}-1$ of these $2^\ell-1$ subspaces then
\begin{equation}
\label{eq:cont2}
\sum_{i=1}^{2^\ell-1} | X_i \cap X_r | = 2^\ell-1 +(2^{\ell-1}-1)(2^{\ell-1}-1)~.
\end{equation}
The solution for the equations (\ref{eq:cont1}) and (\ref{eq:cont2}) is $\lambda_r=1$ which proves the lemma.
\end{proof}

\begin{cor}
\label{cor:S0_typeB}
Let
$$\{ ( {\bf 0},{\bf 0}),({\bf 0},z_1),(x_0,y_0),(x_0,y_1),(x_1,y_2),(x_1,y_3),\ldots,(x_{2^{k-1}-2},y_{2^k-4}),(x_{2^{k-1}-2},y_{2^k-3}) \}$$
and
$$\{ ( {\bf 0},{\bf 0}),({\bf 0},z_2),(u_0,v_0),(u_0,v_1),(u_1,v_2),(u_1,v_3),\ldots,(u_{2^{k-1}-2},v_{2^k-4}),(u_{2^{k-1}-2},v_{2^k-3}) \},$$
be two subspaces of Type B in $\dS_0$.
Then, the $(k-1)$-dimensional subspaces $\{ {\bf 0},x_0,x_1,\ldots,x_{2^{k-1}-2} \}$
and $\{ {\bf 0},u_0,u_1,\ldots,u_{2^{k-1}-2} \}$ are not equal.
\end{cor}
\vspace{0.5cm}

Given the spread $\dS_0$, we define the spread $\dS_i$, $1 \leq i \leq 2^k-2$
as follows.
$$
\dS_i \deff \Bigl\{  \{ ( {\bf 0},{\bf 0}),(x_0,\alpha^i y_0),(x_1, \alpha^i y_1),\ldots,(x_{2^k-2}, \alpha^i y_{2^k-2}) \} ~:~
\{ ( {\bf 0},{\bf 0}),(x_0,y_0),(x_1,y_1),\ldots,(x_{2^k-2},y_{2^k-2}) \} \in \dS_0 \Bigr\}~.
$$
\begin{lemma}
For each $i$, $1 \leq i \leq 2^k-2$, $\dS_i$ is a spread.
\end{lemma}
\begin{proof}
Follows immediately from the following two simple observations. The first one is that if
$\{ ( {\bf 0},{\bf 0}),(x_0,y_0),(x_1,y_1),\ldots,(x_{2^k-2},y_{2^k-2}) \}$
is a $k$-dimensional subspace then also the set
$\{ ( {\bf 0},{\bf 0}),(x_0,\alpha^i y_0),(x_1, \alpha^i y_1),\ldots,(x_{2^k-2}, \alpha^i y_{2^k-2}) \}$
is a $k$-dimensional subspace. The second one is that if
the set $\cF \deff \{ (u_j , v_j ) ~:~ u_j,v_j \in \F_2^k, ~ (u_j ,v_j ) \neq ( {\bf 0},{\bf 0}),~ 0 \leq j \leq 2^{2k}-2  \}$
contains all the $2^{2k}-1$ nonzero elements of $\F_2^k \times \F_2^k$ then the set
$\{ (u_j , \alpha^i v_j ) ~:~  (u_j ,v_j ) \in \cF,~ 0 \leq j \leq 2^{2k}-2  \}$
also contains all the $2^{2k}-1$ nonzero elements of $\F_2^k \times \F_2^k$.
\end{proof}

It is easily verified that
\begin{lemma}
\label{lem:same_type}
For each $0 \leq i \leq 2^k-2$,
if the $k$-dimensional subspace
$$\{ ( {\bf 0},{\bf 0}),(x_0,y_0),(x_1, y_1),\ldots,(x_{2^k-2}, y_{2^k-2}) \}$$
is of Type A (Type B, respectively) then the $k$-dimensional subspace
$$\{ ( {\bf 0},{\bf 0}),(x_0,\alpha^i y_0),(x_1, \alpha^i y_1),\ldots,(x_{2^k-2}, \alpha^i y_{2^k-2}) \}$$
is also of Type A (Type B, respectively).
\end{lemma}

\begin{lemma}
For each $i_1$, $i_2$, such that $0 \leq i_1 < i_2 \leq 2^k-2$, the spreads
$\dS_{i_1}$ and $\dS_{i_2}$ are disjoint.
\end{lemma}
\begin{proof}
By the definition of Type A and Type B, and by the definition of $\dS_j$,
we have that for each $j$, $1 \leq j \leq 2^k-2$,
the number of subspaces of Type A (Type B, respectively) in $\dS_0$ is equal
to the number of subspaces of Type A (Type B, respectively) in $\dS_j$.
Therefore, by Lemma~\ref{lem:S0_two_types}, in $\dS_j$, $1 \leq j \leq 2^k-2$,
there are exactly $2^k-1$ subspaces of Type B and exactly two subspaces of Type~A.
We distinguish now between the two types of subspaces.

\noindent
{\bf Case 1:} Subspaces of Type B.

By Corollary~\ref{cor:S0_typeB}, if
$\{ {\bf 0},x_0,x_1,\ldots,x_{2^{k-1}-2} \}$ is a $(k-1)$-dimensional subspaces
of $\F_2^k$ then there is at most one subspace of the form
$$\{ ( {\bf 0},{\bf 0}),( {\bf 0},z),(x_0,y_0),(x_0,y_1),(x_1,y_2),(x_1,y_3),\ldots,(x_{2^{k-1}-2},y_{2^k-4}),(x_{2^{k-1}-2},y_{2^k-3}) \},$$
in $\dS_0$. By the construction of spread $\dS_j$, $ 1 \leq j \leq 2^k-2$,
from $\dS_0$, we have that the spreads $\dS_{i_1}$ and $\dS_{i_2}$
can have a subspace of Type B in common if for
such a $(k-1)$-dimensional subspace $\{ {\bf 0},x_0,x_1,\ldots,x_{2^{k-1}-2} \}$,
the two $k$-dimensional subspaces
$$\{ ( {\bf 0},{\bf 0}),( {\bf 0},\alpha^{i_1} z),(x_0,\alpha^{i_1} y_0),(x_0, \alpha^{i_1} y_1),(x_1, \alpha^{i_1} y_2),(x_1 , \alpha^{i_1} y_3),\ldots,(x_{2^{k-1}-2},\alpha^{i_1} y_{2^k-4}),(x_{2^{k-1}-2}, \alpha^{i_1} y_{2^k-3}) \}$$
and
$$\{ ( {\bf 0},{\bf 0}),( {\bf 0},\alpha^{i_2} z),(x_0,\alpha^{i_2} y_0),(x_0, \alpha^{i_2} y_1),(x_1, \alpha^{i_2} y_2),(x_1 , \alpha^{i_2} y_3),\ldots,(x_{2^{k-1}-2},\alpha^{i_2} y_{2^k-4}),(x_{2^{k-1}-2}, \alpha^{i_2} y_{2^k-3}) \}$$
are equal. This is clearly impossible since $\alpha^{i_1} z \neq \alpha^{i_2} z$.
Hence, $\dS_{i_1}$ and $\dS_{i_2}$ have distinct subspaces of Type B.

\noindent
{\bf Case 2:} Subspaces of Type A.

A $k$-dimensional subspace of Type A in $\dS_0$ has the form
$$\{ ( {\bf 0},{\bf 0}),(x_0,y_0),(x_1, y_1),\ldots,(x_{2^k-2}, y_{2^k-2}) \}~,$$
where all the $x_j$'s are different.
If $y_r \neq {\bf 0}$ for some $r$, then
$$\{ ( {\bf 0},{\bf 0}),(x_0,\alpha^{i_1} y_0),(x_1,\alpha^{i_1} y_1),\ldots,(x_{2^k-2},\alpha^{i_1} y_{2^k-2}) \}
\neq \{ ( {\bf 0},{\bf 0}),(x_0,\alpha^{i_2} y_0),(x_1,\alpha^{i_2} y_1),\ldots,(x_{2^k-2},\alpha^{i_2} y_{2^k-2})~.$$
By Lemma~\ref{lem:S0_typeA}, not all the $y_j$'s are \emph{zeroes} and
at most one of the subspaces of Type A in $\dS_0$ has the form
$\{ ( {\bf 0},{\bf 0}),(\alpha^0, \alpha^j ),(\alpha^1, \alpha^{j+1} ),\ldots,(\alpha^{2^k-2}, \alpha^{j+2^k-2}) \}$,
for some $j$, $0 \leq j \leq 2^k-2$.
Let
$$\{ ( {\bf 0},{\bf 0}),(x_0, y_0),(x_1,  y_1),\ldots,(x_{2^k-2},  y_{2^k-2}) \}$$
and
$$\{ ( {\bf 0},{\bf 0}),(x_0, v_0),(x_1, v_1),\ldots,(x_{2^k-2}, v_{2^k-2}) \}$$
be the two subspaces of Type A in $\dS_0$.
Assume a $k$-dimensional subspace of Type A
in $\dS_{i_1}$ is equal a $k$-dimensional subspace of Type A in $\dS_{i_2}$. Then
$$\alpha^{i_1} y_\ell = \alpha^{i_2} v_\ell$$
for each $\ell$, $1 \leq \ell \leq 2^k-2$. It implies that for
each $\ell$, $1 \leq \ell \leq 2^k-2$, $\frac{y_\ell}{v_\ell} = \alpha^{i_2-i_1}$, and hence
for all $\ell_1$, $\ell_2$, $0 \leq \ell_1 < \ell_2 \leq 2^k-1$ we have
$\frac{y_{\ell_2}}{v_{\ell_2}} = \frac{y_{\ell_1}}{v_{\ell_1}}$.
We distinguish now between two subcases.

\noindent
{\bf Case 2.1:} Assume that there is no subspace
of the form
$\{ ( {\bf 0},{\bf 0}),(\alpha^0, \alpha^j ),(\alpha^j, \alpha^{j+1} ),\ldots,(\alpha^{2^k-2}, \alpha^{j+2^k-2}) \} ,$
for any $j$, $0 \leq j \leq 2^k-2$, in $\overleftrightarrow{\C}$.
Hence, $\dS_0$ contains the $k$-dimensional subspace
$$\{ ( {\bf 0},{\bf 0}),(\alpha^0, \alpha^0 ),(\alpha^1, \alpha^1 ),\ldots,(\alpha^{2^k-2}, \alpha^{2^k-2}) \}$$
and the second subspace of type A does not has the form
$$\{ ( {\bf 0},{\bf 0}),(\alpha^0, \alpha^j ),(\alpha^1, \alpha^{j+1} ),\ldots,(\alpha^{2^k-2}, \alpha^{j+2^k-2}) \} ,$$
for any $j$, $1 \leq j \leq 2^k-2$.
It implies that there exist $\ell_1$, $\ell_2$,
$0 \leq \ell_1 < \ell_2 \leq 2^k-1$ such that $\frac{y_{\ell_2}}{v_{\ell_2}} \neq \frac{y_{\ell_1}}{v_{\ell_1}}$, a contradiction.

\noindent
{\bf Case 2.2:} Assume that for some $j$, $0 \leq j \leq 2^k-2$, there exists a subspace
of the form
$\{ ( {\bf 0},{\bf 0}),(\alpha^0, \alpha^j ),(\alpha^1, \alpha^{j+1} ),\ldots,(\alpha^{2^k-2}, \alpha^{j+2^k-2}) \} ,$
in $\overleftrightarrow{\C}$.
Hence, the two  subspaces of Type A in $\dS_0$ have the form
$\{ ( {\bf 0},{\bf 0}),(\alpha^0,\alpha^0),(\alpha^1,\alpha^2),
(\alpha^2 ,\alpha^4 ),\ldots,(\alpha^{2^k-2},\alpha^{2^k-3}) \}$
and
$\{ ( {\bf 0},{\bf 0}),(\alpha^0,\alpha^j +\alpha^0),(\alpha^1,\alpha^{j+1} + \alpha^2),
(\alpha^2 ,\alpha^{j+2} +\alpha^4 ),\ldots,(\alpha^{2^k-2},\alpha^{j+2^k-2} +\alpha^{2^k-3}) \}$.
W.l.o.g. we can assume that $x_0=\alpha^0$ and $x_1=\alpha^1$.
It implies that $\frac{y_0}{v_0}=\alpha^j + \alpha^0 \neq \alpha^j + \alpha^1 = \frac{y_1}{v_1}$, a contradiction.

Thus, $\dS_{i_1}$ and $\dS_{i_2}$ are disjoint spreads in $\cG_2(2k,k)$.
\end{proof}

\begin{cor}
There exists a set of $2^k-1$ pairwise disjoint spreads in $\cG_2(2k,k)$.
\end{cor}
\begin{cor}
There exists a set of $2^{k+1}-1$ pairwise disjoint $k$-spreads
in $PG(2k+1,2)$.
\end{cor}

\section{A construction for $q>2$ and $n=2k$}
\label{sec:any_q}

In this section we will describe a construction of two
disjoint spreads in $\cG_q(2k,k)$ for any $q > 2$.
The idea behind the construction will be similar to the one for $q=2$.
But, since we construct only two disjoint spreads, the analysis will
be much simpler. We will start by modifying and generalizing the definition of
the case where $q=2$ for $q \geq 2$.

For a given $X \in \cG_q(k,1)$, let $\V_X^{(n,k)}$ denote the set nonzero
vectors in $\F_q^n$ whose first $k$ entries form any given nonzero vector of $X$.
Let $\V_{\bf 0}^{(n,k)}$ denote a maximal set of $\frac{q^{n-k}-1}{q-1}$ nonzero
vectors in $\F_q^n$ whose first $k$ entries are \emph{zeroes}, for which
any two vectors in the set are linearly independent.
Let $\V_{\bf 0}$ denote the $k$-dimensional subspace spanned by $\V_{\bf 0}^{(n,k)}$.

We consider $k$-dimensional subspaces of three types:
\begin{enumerate}
\item A $k$-dimensional subspace $Y \in \cG_q(2k,k)$ is of Type A if
for each $X \in \cG_q(k,1)$, $Y$ contains
exactly one vector from $\V_X^{(2k,k)}$, and $Y$ does not contain
any vector from $\V_0^{(2k,k)}$.

\item A $k$-dimensional subspace $Y \in \cG_q(2k,k)$ is of Type B if
$Y$ contains exactly one vector from $\V_0^{(2k,k)}$.

\item A $k$-dimensional subspace $Y \in \cG_q(2k,k)$ is of Type C if all the vectors
of $Y$ are contained in $\V_0^{(2k,k)}$.
\end{enumerate}

Throughout this section let $\ell = \frac{q^k-1}{q-1} -1$.
Let $\C_0$ be an $(2k,q^{2k},2(k-1),k)_q$ $\CMRD$.
$\C_0$ is constructed from a linear rank-metric code $\cC$ and therefore
the $k$-dimensional subspace
$\Span{\{ ({\bf 0},{\bf 0}),(x_0,0),(x_1,0),\ldots,(x_{\ell},0) \}}$
is a codeword of $\C_0$. Since the minimum subspace distance of $\C_0$ is $2(k-1)$,
it follows that if
$\Span{\{ ( {\bf 0},{\bf 0}),(x_0,y_0),(x_1,y_1),\ldots,(x_{\ell},y_{\ell}) \}}$ is
another codeword of $\C_0$, then
at most one of $y_i$'s is an all-zero vector. Since  $| \cG_2(k,1)|=\frac{q^k-1}{q-1}$ it
follows from Lemma~\ref{lem:resolv} and Corollary~\ref{cor:MRD_parts} that

\begin{lemma}
The code $\C_0$ has an $(n,q^k,2k,k)_q$ subcode $\C'_0$ which
contains exactly $\frac{q^k-1}{q-1}$ codewords of the form
$$\Span{\{ ( {\bf 0},{\bf 0}),(x_0,y_0),(x_1,y_1),\ldots,(x_{\ell},y_{\ell}) \}},$$
in which exactly one of the $y_i$'s is the all-zero vector.
\end{lemma}

\begin{cor}
\label{cor:Cstructq}
There exists an $(2k,q^k+1,2k,k)_q$ code which contains $\V_0$
as a codeword and for each codeword
$\Span{\{ ( {\bf 0},{\bf 0}),(x_0,y_0),(x_1,y_1),\ldots,(x_{2^k-2},y_{2^k-2}) \}}$
at most one of the $y_i$'s is the all-zero vector.
\end{cor}

Let $\C$ be an $(2k,q^k+1,2k,k)_q$ code as described in Corollary~\ref{cor:Cstructq}.
Let $\overleftrightarrow{\C}$ be the $(2k,q^k+1,2k,k)_q$ code
obtained from $\C$ as follows

$$
\overleftrightarrow{\C} \deff \Bigl\{ \Span{\{ ( {\bf 0},{\bf 0}),(x_0,y_0),(x_1,y_1),\ldots,(x_{\ell},y_{\ell}) \}} ~:~
\Span{\{ ( {\bf 0},{\bf 0}),(y_0,x_0),(y_1,x_1),\ldots,(y_{\ell},x_{\ell}) \}} \in \C  \Bigr\} .
$$

As a consequence of Corollary~\ref{cor:Cstructq} we have
\begin{lemma}
\label{lem:C_two_types_q}
The code $\overleftrightarrow{\C}$ is a spread which consists
of exactly $\frac{q^k-1}{q-1}$ subspaces of Type B and exactly
$q^k +1 - \frac{q^k-1}{q-1}$ subspaces of Type A.
\end{lemma}

\begin{theorem}
There exist at least two disjoint $(2k,q^k+1,2k,k)_q$ codes.
\end{theorem}
\begin{proof}
By Corollary~\ref{cor:MRD_parts}, $\C_0$ can be partitioned into $q^k$   $(2k,q^k,2k,k)_q$
codes. Since $q^k > q^k +1 - \frac{q^k-1}{q-1}$ it follows that at least
one of these $q^k$ codes does not contain any of the $q^k +1 - \frac{q^k-1}{q-1}$
subspaces of Type A which are contained in $\overleftrightarrow{\C}$.
Let $\C'$ be this code. $\C' \cup \V_0$ is a $(2k,q^k+1,2k,k)_q$ code
which contains $q^k$ subspaces of Type A and one subspace of Type C.
Therefore, $\overleftrightarrow{\C}$ and $\C'$ are disjoint.
\end{proof}
\begin{cor}
There exist two disjoint spreads in $\cG_q (2k,k)$, $q > 2$.
\end{cor}
\begin{cor}
There exist two disjoint $k$-spreads in PG($2k+1,q$), $q > 2$.
\end{cor}

\section{A recursive construction}
\label{sec:recursive}

Let $n=\ell k$, where $\ell \geq 2$.
Let $\dS_i$, $0 \leq i \leq M-1$, be a set of $M$ pairwise disjoint
spreads in $\cG_q(n,k)$. We will describe a construction
for $M$ pairwise disjoint spreads in $\cG_q( n+k,k)$.

First we will define a \emph{partial Grassmannian} $\cG_q(n_1,n_2,k)$,
$n_1 > n_2 \geq k$, as the set of all $k$-dimensional subspaces from $\F_q^{n_1}$
which are not contained in a given $n_2$-dimensional subspace $U$ of $\F_q^{n_1}$.
It can be readily verified that $\V^{(n,k)}$ is a partial Grassmannian
$\cG_q(n,n-k,k)$, where $\V_0^{(n,k)}$ is the $(n-k)$-dimensional
subspace $U$.
A spread in $\cG_q(n_1,n_2,k)$ is a set $\dS$ of pairwise disjoint $k$-dimensional
subspaces from $\cG_q(n_1,n_2,k)$ such that each nonzero element of $\F_q^{n_1} \setminus U$
is contained in exactly one element of~$\dS$. A parallelism of $\cG_q(n_1,n_2,k)$
is a set of pairwise disjoint spreads in $\cG_q(n_1,n_2,k)$ such that each $k$-dimensional
subspace of $\cG_q(n_1,n_2,k)$ is contained in exactly one of the spreads.
Beutelspacher~\cite{Beu90} proved that if $k=2$ then such a parallelism exists
if $n_2 \geq 2$, $n_1 - n_2 = 2^i$, for all $i \geq 1$ and any $q > 2$.
If $k=2$ and $q=2$ then such a parallelism exists
if and only if $n_2 \geq 3$ and $n_1 - n_2$ is even.

In this section we are going to extend this results for $k > 2$.
Based on these parallelisms we
will present a recursive construction for pairwise disjoint spreads
in $\cG_q(n,k)$, where $k$ divides~$n$ and $n > k$.

The following structure defined in~\cite{EtSi13} is
the key for our construction.
A \emph{subspace transversal  design} of groupsize $q^{n-k}$,
block dimension $k$, and \emph{strength}~$t$, denoted by
$\text{STD}_q (t, k, n-k)$, is a triple
$(\V,\mathbb{G},\mathbb{B})$, where
$\V$ is a set of points,
$\mathbb{G}$ is a set of groups, and $\mathbb{B}$ is
a set of blocks. These three sets must satisfy the following five
properties:

\begin{enumerate}
\item $\V$ is a set of size $\frac{q^k-1}{q-1} q^{n-k}$ (the \emph{points}).
$\bigcup_{X \in \cG_q(k,1)} \V_X^{(n,k)}$ is used as the set of points $\V$.

\item $\mathbb{G}$ is a partition of $\V$ into
$\frac{q^k-1}{q-1}$ classes of size $q^{n-k}$ (the \emph{groups});
the groups which are used are defined by
$\V_X^{(n,k)}$, $X \in \cG_q(k,1)$.

\item $\mathbb{B}$ is a collection of $k$-dimensional
subspaces of $\F_q^n$ which contain nonzero vectors
only from $\mathbb{V}^{(n,k)}$ (the
\emph{blocks});

\item each block meets each group in exactly one point;

\item every $t$-dimensional subspace (with points from $\mathbb{V}$) which
meets each group in at most one point is contained in exactly
one block.
\end{enumerate}

An $\text{STD}_q(t, k, m)$ is \emph{resolvable} if the set $\B$
can be partitioned into sets $\B_1,...,\B_s$,
where each vector of $\mathbb{V}^{(n,k)}$ is contained in exactly one block of
each $\B_i$, $1 \leq i \leq s$. The sets $\B_1,...,\B_s$ are called
\emph{parallel classes}. The following theorem was established in~\cite{EtSi13}.

\begin{theorem}
\label{thm:MRD=STD_Q}
The codewords of an $(n,k,\delta)_q$
$\CMRD$ form the blocks of a resolvable
$\text{STD}_q(k-\delta+1,k, n-k)$, with the set of groups $\V_X^{(n,k)}$,
$X\in \cG_q(k,1)$.
\end{theorem}

Theorem~\ref{thm:MRD=STD_Q} is the key for our constructions.
A resolvable STD$_q(k,k,n-k)$ consists of $q^{(n-k)(k-1)}$
spreads of $\V^{(n,k)}$, i.e. a parallelism in $\cG_q(n,n-k,k)$.
A resolvable STD$_q(k,k,n-k)$ is obtained from an $(n,k,1)_q$ $\CMRD$,
which is constructed from a $[k \times (n-k),(n-k)k,1]_q$ MRD code.
Thus, we have

\begin{theorem}
\label{thm:parallelism}
If $k=n_1 - n_2$ then there exists a parallelism in $\cG_q(n_1,n_2,k)$.
\end{theorem}

If there exists $M$ pairwise disjoint spreads in $\cG_q(n-k,k)$ then
they can be combined with~$M$ pairwise disjoint spreads in $\cG_q(n,n-k,k)$ which
exist by Theorem~\ref{thm:parallelism} to obtain the following theorem.

\begin{theorem}
If there exist $M$ pairwise disjoint spreads in $\cG_q(n-k,k)$ then
there exist $M$ pairwise disjoint spreads in $\cG_q(n,k)$.
\end{theorem}
\begin{proof}
By Theorem~\ref{thm:parallelism}, there exists $M$ pairwise disjoint spreads
in $\cG_q(n,n-k,k)$, in which the removed $(n-k)$-dimensional subspace
is isomorphic to $\cG_q(n-k,k)$. Let $\dS_1 , \dS_2 , \ldots , \dS_M$, be these spreads.
Let $\dT_1, \dT_2 , \ldots , \dT_M$ be the $M$ pairwise disjoint spreads in $\cG_q(n-k,k)$.
Then $\dS_1 \cup \dT_1 , \dS_2 \cup \dT_2 , \ldots , \dS_M \cup \dT_M$ is a set of
$M$ pairwise disjoint spreads in $\cG_q(n,k)$.
\end{proof}
\begin{cor}
There exists a set of $2^k-1$ pairwise disjoint spreads in $\cG_2(n,k)$ if $n >k$
and $k$ divides~$n$.
\end{cor}
\begin{cor}
There exist two pairwise disjoint spreads in $\cG_q(n,k)$ if $n >k$
and $k$ divides $n$.
\end{cor}
\begin{cor}
There exists a set of $2^{k+1}-1$ pairwise disjoint $k$-spreads in PG($n,2$) if $n >k$
and $k+1$ divides $n+1$.
\end{cor}
\begin{cor}
There exist two pairwise disjoint spreads $k$-spreads in PG($n,q$) if $n >k$
and $k+1$ divides $n+1$.
\end{cor}

\section{Conclusions and problems for future research}
\label{sec:conclusion}

Finding a $k$-parallelism in PG($n,q$) is an extremely difficult
problem. If $k > 1$ then only one such parallelism is known. The goal
of this paper was to direct the research for
the following slightly easier question. What is the maximum number of
pairwise disjoint $k$-spreads in PG($n,q$)? This number can be greater than one
only if $n >k$ and
$k+1$ divides $n+1$ which is the sufficient and necessary
condition for the existence of $k$-spreads in PG($n,q$).
We proved that two such pairwise disjoint $k$-spreads
always exist. If $q=2$ then we proved the existence
of $2^{k+1}-1$ pairwise disjoint $k$-spreads. We also proved
that if $k+1$ divides $n_1+1$ and $n_2+1$, and $n_1 > n_2 >k$,
then there exist a $k$-parallelism in the partial space of
dimension $n_1$ from which an $n_2$-subspace was removed.
There are many interesting open problems in this topic.
We will state them in an increasing order of difficulty
by our opinion, from the
easiest one to the most difficult one.

\begin{enumerate}
\item
For any $q >2$ and $k \geq 1$, improve the lower bounds on the number of pairwise disjoint
$k$-spreads in PG($n,q$).

\item
For any $k > 1$, improve the lower bounds on the number of pairwise disjoint
$k$-spreads in PG($n,2$).

\item
Find nontrivial necessary conditions for the existence of a parallelism in
$\cG_q(n_1,n_2,k)$.

\item
Find new parameters for which there exists a parallelism in
$\cG_q(n_1,n_2,k)$.

\item
For a power of a prime $q>2$, find new parameters for which there
exists a 1-parallelism in PG($n,q$).

\item
For $k>1$ and any power of a prime $q$, starting
with $q=2$, find new parameters for which there
exists a $k$-parallelism in PG($n,q$).

\item
For and $q$ and $k>1$, find an infinite family of $k$-parallelisms in PG($n,q$).

\end{enumerate}

%%%%%%%%%%%%%%%%%%%%%%%%%%%%%%%%%%%%%%%%%%%%%%%%%%%%%%%%%%%%%%%%%%%%%%
%%%%%%%%%%%%%%%%%%%%%%%%%%%%%%%%%%%%%%%%%%%%%%%%%%%%%%%%%%%%%%%%%%%%%%
%%%%%%%%%%%%%%%%%%%%%%%%%%%%%%%%%%%%%%%%%%%%%%%%%%%%%%%%%%%%%%%%%%%%%%
%%%%%%%%%%%%%%%%%%%%%%%%%%%%%%%%%%%%%%%%%%%%%%%%%%%%%%%%%%%%%%%%%%%%%%

%\bibliography{allbib}

\end{document}